\documentclass[11pt]{amsart}

\usepackage[pdftex]{graphicx}
\usepackage{amsmath, amsfonts, amssymb, amsthm}
\usepackage{enumerate}
\usepackage{tikz}
\usepackage{setspace}

\DeclareMathOperator{\dg}{d}

\newtheorem{theorem}{Theorem}
\newtheorem{lemma}[theorem]{Lemma}

\newtheorem*{question}{Question}
\theoremstyle{remark}

\theoremstyle{definition}

\begin{document}

\title{Gerechte Designs with Rectangular Regions}

\author{J. Courtiel}
\address{ENS Cachan Bretagne, Rennes, France}
\email{julien.courtiel@eleves.bretagne.ens-cachan.fr}

\author{E. R. Vaughan}
\address{Queen Mary, University of London, UK}
\email{e.vaughan@qmul.ac.uk}

\date{\today}

\maketitle

\begin{abstract} A \emph{gerechte framework} is a partition of an $n \times n$ array into $n$ regions of $n$ cells each. A \emph{realization} of a gerechte framework is a latin square of order $n$ with the property that when its cells are partitioned by the framework, each region contains exactly one copy of each symbol. A \emph{gerechte design} is a gerechte framework together with a realization.

We investigate gerechte frameworks where each region is a rectangle. It seems plausible that all such frameworks have realizations, and we present some progress towards answering this question. In particular, we show that for all positive integers $s$ and $t$, any gerechte framework where each region is either an $s \times t$ rectangle or a $t\times s$ rectangle is realizable. \end{abstract}

\section{Introduction}

A \emph{latin square} of order $n$ is an $n \times n$ array in which each of the symbols $1, \dots, n$ appears  once in each row and once in each column. A \emph{gerechte framework} of order $n$ is a partition of the cells of an $n \times n$ array into $n$ regions each containing $n$ cells. We say that a latin square \emph{realizes} a gerechte framework if, when the cells of the square are partitioned by the framework, each symbol appears once in each region. Such a latin square is said to be a \emph{realization} of the framework, and frameworks with realizations are said to be \emph{realizable}.

A \emph{gerechte design} consists of a gerechte framework together with a latin square that realizes it. Gerechte designs were introduced in 1956 by Behrens \cite{behrens} who suggested their use in agricultural experiments (see e.g. \cite{bcc}). They have been rediscovered in recent years, due to the popularity of the puzzle game \emph{Sudoku}, invented by Howard Garns in 1979 \cite{bcc}.

It seems natural to ask which gerechte frameworks are realizable. In fact, given a gerechte framework, the problem of deciding whether it is realizable is NP-complete (\cite{vaughan}, answering a question of \cite{cgcs2007}). However, in this article we shall restrict our attention to gerechte frameworks where each region is a rectangle. In \cite{vaughan} it was asked whether all such gerechte frameworks can be realized. While we are not yet able to answer this question, we can present some progress towards resolving it. Our main result is the following.

\begin{figure}
\centering \includegraphics{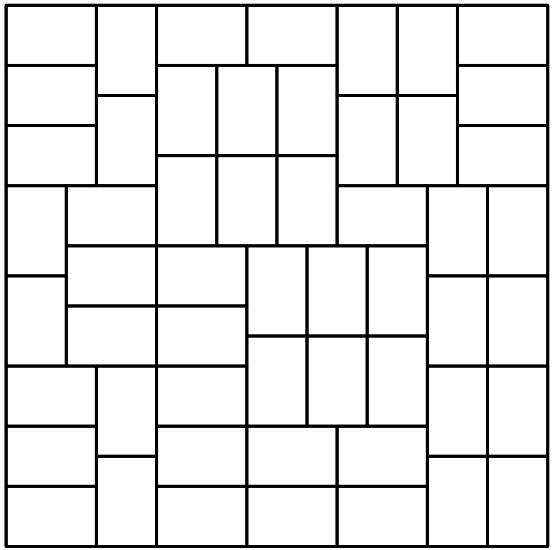}
\caption{A gerechte framework of order 54 with $6 \times 9$ and $9 \times 6$ rectangular regions. (Individual cells are not shown.)} \label{g54empty}
\end{figure}

\begin{theorem} Any gerechte framework for which each region is either an $s \times t$ rectangle or a $t\times s$ rectangle is realizable. \label{main} \end{theorem}

A gerechte framework where each region is an $s \times t$ rectangle has a simple structure, and it is not hard to show that it is realizable (see Theorem \ref{simplest}). But allowing both $s \times t$ and $t \times s$ rectangles (think ``landscape'' and ``portrait'') makes the situation more complicated. An example of a gerechte framework of order 54 with $6 \times 9$ and $9 \times 6$ regions is given in Figure \ref{g54empty}.

In this article, we shall make frequent use of Hilton's Amalgamated Latin Squares Theorem \cite{hilton}, to show that latin squares can be constructed that realize particular frameworks. In order for our article to be as self-contained as possible, we provide a proof of Hilton's theorem. Apart from using the well-known theorem of K\"onig \cite{bm}, the article is self-contained.

In Section~2 we give the background results on edge-colourings and amalgamated latin squares. We prove two special cases of  Theorem \ref{main} in Section~3, and in Section~4 we give the proof of Theorem \ref{main}. In Section~5 we give two generalizations of Theorem \ref{main}.

\section{Preliminaries}

\subsection{Edge-Colourings of Bipartite Multigraphs}

An edge-colouring of a multigraph $G$ is a map from $E(G) \rightarrow C$, where $C$ is a set of colours. If an edge-colouring has the property that no two adjacent edges are assigned the same colour, it is said to be \emph{proper}. The least number of colours needed for a proper edge-colouring of $G$ is called the \emph{chromatic index} of $G$, which is denoted $\chi'(G)$. The maximum degree of a multigraph $G$ is denoted $\Delta(G)$.

\begin{theorem} \emph{(K\"onig's Theorem \cite{bm})} Let $G$ be a bipartite multigraph. Then $\chi'(G) = \Delta(G)$. \label{konig} \end{theorem}

We are not only interested in proper colourings. The following theorem is a special case of a more general theorem of de Werra \cite{dewerra}, which says that for any bipartite multigraph, there is an edge-colouring that is \emph{equitable}; which is to say that among the edges incident with any particular vertex, the colours are shared out as equally as possible. The following special case can be deduced from K\"onig's Theorem.

\begin{theorem} Let $G$ be a bipartite multigraph, and $k$ a positive integer. If each vertex has degree a multiple of $k$, then there is an edge-colouring of $G$ with $k$ colours, such that for each vertex $v \in V(G)$, each colour appears $\dg(v)/k$ times on the edges incident with $v$. \label{vkonig} \end{theorem}

\begin{proof} Let $G'$ be the bipartite multigraph obtained from $G$ by splitting each vertex $v \in V(G)$ into $\dg(v)/k$ new vertices, sharing out the edges in any way so that each new vertex has degree $k$. By Theorem \ref{konig} we can give $G'$ a proper edge-colouring with $k$ colours. But then we can colour each edge of $G$ the same colour as the corresponding edge in $G'$, to obtain the required edge-colouring. \end{proof}

\begin{figure}
\centering \begin{tikzpicture}
[matrix of nodes/.style={minimum size=7mm, execute at begin cell=\node\bgroup, execute at end cell=\egroup; }]
\draw[gray, dashed, step=7mm] (0,0) grid (4.2,4.2);
\draw[black, step=7mm] (0,0) rectangle (4.2,4.2)
(1.4,0) rectangle (2.8, 4.2)
(0, 2.1) rectangle (4.2, 3.5);
\node [matrix, matrix of nodes] at (2.1,2.1) {
1 & 4 & 2 & 3 & 5 & 6 \\
3 & 5 & 4 & 6 & 1 & 2 \\
4 & 6 & 1 & 5 & 2 & 3 \\
2 & 3 & 5 & 4 & 6 & 1 \\
6 & 1 & 3 & 2 & 4 & 5 \\
5 & 2 & 6 & 1 & 3 & 4 \\ };
\draw (2.1,0) node[below] {$L$};
\begin{scope}[xshift=6cm]
\draw[black,step=7mm] (0,0) rectangle (4.2,4.2)
(1.4,0) rectangle (2.8, 4.2)
(0, 2.1) rectangle (4.2, 3.5);
\node [matrix, matrix of nodes] at (2.1,2.1) {
1 & 3 & 2 & 3 & 4 & 4 \\
3 & 4 & 3 & 4 & 1 & 2 \\
3 & 4 & 1 & 4 & 2 & 3 \\
2 & 3 & 4 & 3 & 4 & 1 \\
4 & 1 & 3 & 2 & 3 & 4 \\
4 & 2 & 4 & 1 & 3 & 3 \\ };
\draw (2.1,0) node[below] {$L^*$};
\end{scope}
\end{tikzpicture}
\caption{A latin square $L$ (left), and $L^*$ (right), an amalgamation of $L$.}
\label{amalgamated}
\end{figure}

\subsection{Outline and Amalgamated Latin Squares}

A \emph{composition} of $n$ is an ordered tuple of positive integers that sum to $n$. Suppose $L$ is a latin square of order $n$, and $S = (p_1, \dots ,p_s)$, $T = (q_1, \dots, q_t)$ and $U = (r_1, \dots, r_u)$ are three compositions of $n$. Then the $(S,T,U)$-\emph{amalgamated latin square} $L^*$ is the $s \times t$ array whose cells contain symbols from $1, \dots, u$, where the number of copies of the symbol $k$ in the cell $(i,j)$ equals the number of occurrences of a symbol from
\[ \{r_1 + \dots +r_{k-1} + 1, \dots, r_1 + \dots + r_k\} \]
in one of the cells $(\alpha, \beta)$ of $L$ where
\[ \alpha \in \{p_1 + \dots + p_{i-1} + 1, \dots, p_1 + \dots + p_i\} \]
and
\[ \beta \in \{q_1 + \dots + q_{j-1} + 1, \dots, q_1 + \dots + q_j\}. \]
$L^*$ is said to be the $(S,T,U)$-\emph{amalgamation} of $L$.

In the example shown in Figure \ref{amalgamated}, $L$ is a latin square of order 6 and $L^*$ is the $(S,T,U)$-amalgamation of $L$, where $S=(1,2,3)$, $T=(2,2,2)$ and $U=(1,1,2,2)$.

\begin{figure}
\centering \includegraphics{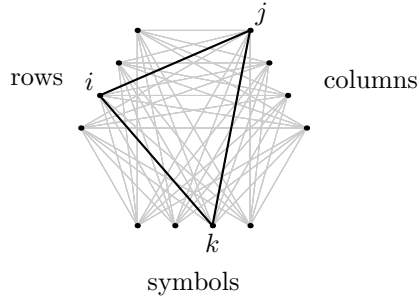}
\caption{A latin square can be regarded as a decomposition of $K_{n,n,n}$ into edge-disjoint triangles.} \label{tripartite}
\end{figure}

Note that in an $(S,T,U)$-amalgamated latin square the following conditions hold:

\begin{enumerate}[(i)]
\item symbol $\sigma_k$ occurs $p_ir_k$ times in row $i$.
\item symbol $\sigma_k$ occurs $q_jr_k$ times in column $j$.
\item cell $(i,j)$ contains $p_iq_j$ symbols (counting repetitions).
\end{enumerate}

Suppose $S = (p_1, \dots, p_s)$, $T = (q_1, \dots, q_t)$ and $U = (r_1, \dots, r_u)$ are three compositions of $n$. An $(S,T,U)$-\emph{outline latin square} is an $s \times t$ array, whose cells contain symbols from $1, \dots, u$, such that the above conditions (i), (ii) and (iii) hold. In other words, an outline latin square satisfies the same three numerical constraints as an amalgamated latin square, but we do not assume that it was obtained from a latin square by amalgamation. However, the following theorem states that every outline latin square can in fact be obtained by amalgamating a latin square.

\begin{theorem}[Hilton \cite{hilton}] Each outline latin square is an amalgamated latin square. \label{alst} \end{theorem}

It is sometimes useful to view a latin square as a triangle decomposition of the complete tripartite graph $K_{n,n,n}$, in which the triangle $(i,j,k)$ belongs to the decomposition if and only if cell $(i,j)$ of the latin square contains the symbol $k$. (See Figure \ref{tripartite}.) So an amalgamated latin square is obtained from a latin square by amalgamating vertices of $K_{n,n,n}$ (where vertices can be amalgamated only if they are in the same part). With this viewpoint, it is clear that conditions (i), (ii) and (iii) are symmetric.

\begin{proof}[Proof of Theorem \ref{alst}] Let $M$ be an $(S,T,U)$-outline latin square of order $n$, where $S = (p_1, \dots, p_s)$, $T = (q_1, \dots, q_t)$ and $U = (r_1, \dots, r_u)$ are three compositions of $n$. If $S = T = U = (1, \dots, 1)$ then $M$ is a latin square and there is nothing to prove. So suppose that at least one of $S$, $T$, $U$ is not equal to $(1, \dots, 1)$. In view of the symmetry, we may assume that $S \ne (1, \dots, 1)$ and specifically that $p_s \ne 1$.

We shall show how we can obtain from $M$ an $(S',T,U)$-outline latin square $M'$, where $S' = (p_1, \dots, p_{s-1}, 1, \dots, 1)$. (In other words, we shall split row $s$ into $p_s$ rows.) By repeating this procedure on all the rows, and then on the columns and symbols, we can obtain an $(S'',T'',U'')$-outline latin square, where $S'' = T'' = U'' = (1, \dots, 1)$, or, in other words, a latin square. It will then follow that $M$ is an $(S,T,U)$-amalgamated latin square.

Let $G$ be the bipartite graph with vertex sets $\{c_1, \dots, c_t\}$ and $\{\sigma_1, \dots, \sigma_u\}$, representing the columns and symbols respectively, where there are $l$ edges between $c_j$ and $\sigma_k$ if the cell $(s, j)$ contains $l$ copies of symbol $k$. By condition (iii), vertex $c_j$ has degree $p_s q_j$, and by condition (i), vertex $\sigma_k$ has degree $p_s r_k$. By Theorem \ref{vkonig}, we can give $G$ an edge-colouring with $p_s$ colours $k_1, \dots, k_{p_s}$, so that at each vertex, each colour appears the same number of times on the incident edges. Specifically, at vertex $c_j$, each colour appears $q_j$ times, and at vertex $\sigma_k$, each colour appears $r_k$ times.

Now we can construct the $(S',T,U)$-outline latin square $M'$. The first $s-1$ rows will be the same as in $M$. Rows $s, \dots, s + p_s - 1$ will be constructed as follows. For row $s + i - 1$, where $1 \le i \le p_s$, if $l$ is the number of edges coloured $k_i$ between $c_j$ and $\sigma_k$ in $G$, we place $l$ copies of the symbol $\sigma_k$ in cell $(s + i - 1, j)$.

We must now verify that $M'$ satisfies the conditions (i), (ii) and (iii). We have not altered the number of times each symbol appears in any column, so condition (ii) will certainly hold. Condition (i) will certainly hold for the first $s-1$ rows, so we just need to check it for the remaining rows. In the edge-colouring of $G$, the vertex $\sigma_k$ is incident with $r_k$ edges of each colour, so will appear $r_k$ times in each of the rows $s, \dots, s + p_s - 1$, so condition (i) is satisfied.

Condition (iii) will certainly hold for cells in the first $s-1$ rows. The vertex $c_j$ in $G$ is incident with $q_j$ edges of each colour, so each cell in the last $p_s$ rows of column $j$ will contain $q_j$ symbols, so condition (iii) is satisfied. Hence $M'$ is indeed an $(S',T,U)$-outline latin square, and the proof of Theorem \ref{alst} is complete. \end{proof}

\section{Gerechte Designs with Rectangular Regions: The~Simplest~Cases}

The first case we consider is that of gerechte frameworks where each region is an $s \times t$ rectangle. Using Theorem~\ref{alst}, we can show that all such frameworks are realizable.

\begin{theorem} Let $F$ be a gerechte framework of order $n = st$, where each region is an $s \times t$ rectangle. Then $F$ is realizable. \label{simplest} \end{theorem}

\begin{proof} Construct a $t \times s$ array $M$ where each cell contains \textit{all} of the symbols $1, \dots, n$. Thus each row of $M$ contains each symbol $s$ times, each column contains each symbol $t$ times, and each cell contains $n = st$ symbols. So $M$ is an $(S,T,U)$-outline latin square, where $S=(s, \dots, s)$, $T=(t, \dots, t)$ and $U=(1, \dots, 1)$. By Theorem \ref{alst}, $M$ is an amalgamation of a latin square $L$. Assume that the cells of $L$ are partitioned by $F$. Then because $M$ is the $(S,T,U)$-amalgamation of $L$, each region of $L$ contains all of the symbols $1, \dots, n$, and so $L$ is a realization of $F$. \end{proof}

Suppose we have a framework $F$ where each region is an $s \times t$ rectangle or a $t \times s$ rectangle, and there is some $k > 1$ for which $k \mid s$ and $k \mid t$. Then the \textit{reduced framework} $F/k$ is a partition of the cells of an $n/k \times n/k$ array that is obtained from $F$ by amalgamating $k \times k$ squares (see Figure \ref{ger3}).

\begin{theorem} Let $F$ be a gerechte framework of order $n = st$, where each region is an $s \times t$ rectangle, or a $t \times s$ rectangle, with the condition that $t=cs$ for some integer $c$. Then $F$ is realizable. \label{divides} \end{theorem}

\begin{proof} Construct a $t \times s$ array $M$ where cell $(i,j)$ contains the symbol from $1, \dots, c$ that is congruent to $i + j - 1$ modulo $c$. Assume that the cells of $M$ are partitioned by $F/s$. Then each region contains $c$ cells, and because the symbols in adjacent cells differ by 1 modulo $c$, each region will in fact contain each of the symbols $1, \dots, c$ exactly once. Let $M'$ be the array obtained from $M$ by replacing each symbol $i$ with the $s^2$ symbols
\[ (i-1)s^2+1, \dots, is^2. \]
So when the cells of $M'$ are partitioned by $F/s$, each region contains one copy of each of the symbols $1, \dots, n$. Moreover, $M'$ is an $(S,T,U)$-outline latin square, where $S=(s, \dots, s)$, $T=(s, \dots, s)$ and $U=(1, \dots, 1)$. By Theorem \ref{alst}, $M'$ is an amalgamation of a latin square $L$. Assume that the cells of $L$ are partitioned by $F$. Then because $M'$ is the $(S,T,U)$-amalgamation of $L$, each region of $L$ contains all of the symbols $1, \dots, n$, and so $L$ is a realization of $F$.  \end{proof}

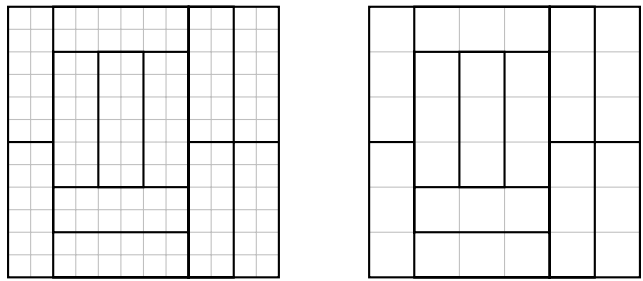
\begin{figure}
\centering
\begin{tikzpicture}
\begin{scope}[scale=0.6]
\draw[lightgray, thin, step=0.5] (0,0) grid (6,6);
\draw[black, thick] (0,0) rectangle (6, 6) (1,1) rectangle (4,6)
(1,0) rectangle (4,5) (4,0) rectangle (5,6) (2,2) rectangle (3,5)
(0,3)--(1,3) (4,3)--(6,3) (1,2)--(4,2);
\end{scope}
\begin{scope}[scale=0.6, xshift=8cm]
\draw[lightgray, thin, step=1] (0,0) grid (6,6);
\draw[black, thick] (0,0) rectangle (6, 6) (1,1) rectangle (4,6)
(1,0) rectangle (4,5) (4,0) rectangle (5,6) (2,2) rectangle (3,5)
(0,3)--(1,3) (4,3)--(6,3) (1,2)--(4,2);
\end{scope}
\end{tikzpicture}
\caption{A gerechte framework $F$ with $2 \times 6$ and $6 \times 2$ regions (left), and the framework $F/2$, obtained from $F$ by amalgamating $2 \times 2$ squares (right).}
\label{ger3}
\end{figure}

Note that in the proof it was not necessary to make use of the structure of the framework $F$. In fact $L$ is a realization of any framework that satisfies the conditions given.

We can illustrate Theorem \ref{divides} with the following example. Figure \ref{ger3} (left) shows a gerechte framework $F$ with $2 \times 6$ and $6 \times 2$ regions. We can amalgamate $2 \times 2$ squares of $F$ to give a framework $F/2$ that has $1 \times 3$ and $3 \times 1$ regions, as shown in Figure \ref{ger3} (right). We can then fill each cell $(i,j)$ with the symbol from $1,2,3$ that is congruent to $i + j - 1 \mod 3$, as illustrated in Figure \ref{ger4}. When this array is partitioned by $F/2$, each region contains one copy of each symbol. Finally, we can replace the symbols $1,2,3$ with the sets $\{1,2,3,4\}$, $\{5,6,7,8\}$ and $\{9,10,11,12\}$ respectively to get an amalgamated latin square, and apply Theorem \ref{alst} to get a latin square $L$ that realizes $F$.

\begin{figure}
\centering
\begin{tikzpicture}
[matrix of nodes/.style={minimum size=6mm, execute at begin cell=\node\bgroup, execute at end cell=\egroup; }]
\begin{scope}[scale=0.6]
\node [matrix, matrix of nodes] at (3,3) {
1 & 2 & 3 & 1 & 2 & 3 \\
2 & 3 & 1 & 2 & 3 & 1 \\
3 & 1 & 2 & 3 & 1 & 2 \\
1 & 2 & 3 & 1 & 2 & 3 \\
2 & 3 & 1 & 2 & 3 & 1 \\
3 & 1 & 2 & 3 & 1 & 2 \\ };
\draw[lightgray, thin, step=1] (0,0) grid (6,6);
\draw[black, thick] (0,0) rectangle (6, 6) (1,1) rectangle (4,6)
(1,0) rectangle (4,5) (4,0) rectangle (5,6) (2,2) rectangle (3,5)
(0,3)--(1,3) (4,3)--(6,3) (1,2)--(4,2);
\end{scope}
\end{tikzpicture}
\caption{An example of an $M$ from Theorem \ref{divides}.}
\label{ger4}
\end{figure}

\section{Gerechte Designs with $s \times t$ and $t \times s$ regions: the~general~case.}

In this section we shall prove Theorem \ref{main}. First some definitions. We say that a region $R$ \textit{begins} in row $i$ if the uppermost cells in $R$ below to row $i$. Likewise we say that a region begins in column $j$ if the leftmost cells in $R$ belong to column $j$.

\begin{lemma} Let $F$ be a gerechte framework of order $n$, where each region is either an $s \times t$ rectangle or a $t \times s$ rectangle. Suppose $s = s'k$ and $t = t'k$, where $k = \gcd(s,t)$. Then the number of $s \times t$ regions that begin in row $i$, for each $1 \le i \le n$, is a multiple of $s'$. \label{paving} \end{lemma}

\begin{proof} For $1 \le i \le n$, let $n_i$ be the number of $s \times t$ regions that begin in row~$i$, and let $m_i$ be the number of $t \times s$ regions that begin in row $i$. Since each cell in row $1$ is in the uppermost row of some region, $t n_1 + s m_1 = n = st$. Dividing through by $k$, we have $t' n_1 + s' m_1 = s't$, and so $t' n_1 = s'(t - m_1)$. Since $s'$ and $t'$ are relatively prime, we must have $s' \mid n_1$.

\begin{figure}
\centering \includegraphics{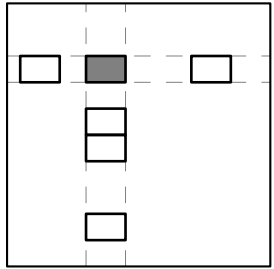}
\caption{Horizontally and vertically aligned regions.} \label{aligned}
\end{figure}

Now assume that $s'$ divides $n_1, \dots, n_{i-1}$ and consider row $i$. The cells in row $i$ that are in $s \times t$ regions, are in regions that begin in rows $i - s + 1, \dots, i$. If $j$ is the number of $t \times s$ regions that have cells in row $i$, then
\[ st = t(n_{i - s + 1} + \dots + n_i) + sj.\]
Dividing through by $k$ we have
\[s't = t'(n_{i - s + 1} + \dots + n_i) + s'j,\] and so
\[t' n_i = s'(t-j) - t'(n_{i - s + 1} + \dots + n_{i-1}))\]
Since each term on the right-hand side is divisible by $s'$, and $s'$ and $t'$ are relatively prime, we have $s' \mid n_i$.
\end{proof}

A region $Q$ is said to be \textit{horizontally aligned} with a region $R$ if $Q$ and $R$ begin in the same row, and have the same height. Similarly, $Q'$ is said to be \textit{vertically aligned} with $R$ if $Q'$ and $R$ begin in the same column, and have the same width (see Figure \ref{aligned}). We are now ready to prove Theorem \ref{main}.

\begin{proof}[Proof of Theorem \ref{main}] Let $F$ be a gerechte framework of order $n$, where each region is either an $s \times t$ rectangle or a $t \times s$ rectangle. Suppose $s = s'k$ and $t = t'k$, where $k = \gcd(s,t)$.

We shall give a procedure for constructing a latin square $L$ that realizes $F$. The first step is to construct an $n/k \times n/k$ array $M$, where each cell contains one symbol from $1, \dots, s't'$, each symbol appears $k$ times in each row and column, and when partitioned by $F/k$, each symbol appears once in each region.

$M$ will be constructed from an empty array, by first filling the cells in the $s' \times t'$ regions, and then the cells in the $t' \times s'$ regions. These two steps can be performed entirely independently. A \emph{slice} is a maximal set of cells of $M$ that are all in the same row and region.

Let $R$ be an $s' \times t'$ region.  We can fill $R$ by placing one copy of each of the symbols $1, \dots, s't'$ in $R$. The $s' \times t'$ regions which are horizontally aligned with $R$ can then be filled by cyclically permuting the slices of $R$, as illustrated in Figure~\ref{slices1}.

By repeating this procedure, all the cells in the $s' \times t'$ regions can be filled, with each region containing all of the symbols $1, \dots, s't'$. Moreover, due to the way that that the slices are permuted, each row will contain the same number of copies of each symbol. However, it will in general not be the case that each column contains the same number of copies to each symbol.

\begin{figure}
\centering \includegraphics{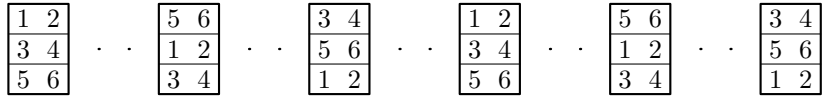} 
\caption{Filling $R$ and horizontally-aligned regions.} \label{slices1}
\end{figure}

Consider $R$ and the $s' \times t'$ regions which are vertically aligned with $R$. By Lemma~\ref{paving}, the number of $s' \times t'$ regions vertically aligned with $R$, including $R$ itself, is a multiple of $t'$. So we can take, from these regions, the $t'$ slices that contain $1, \dots, t'$, and cyclically permute them, as illustrated in Figure~\ref{slices2}. This can also be done for the slices that contain $t' + 1, \dots, 2t'$, and so on. This procedure can be repeated for all the $s' \times t'$ regions. Once this is done, each column will contain the same number of copies of each symbol.

The above procedure can then be performed on the $t' \times s'$ regions. In this way, $M$ can be completely filled, with each symbol appearing the same number of times (in fact $k$ times) in each row and column, and each symbol appearing once in each region.

Let $M'$ be the array obtained from $M$ by replacing each symbol $i$ with the $k^2$ symbols
\[ (i-1)k^2+1, \dots, ik^2. \]
So when the cells of $M'$ are partitioned by $F/k$, each region contains each of the symbols $1, \dots, n$ exactly once. Moreover, $M'$ is an $(S,T,U)$-outline latin square, where $S=(k, \dots, k)$, $T=(k, \dots, k)$ and $U=(1, \dots, 1)$. By Theorem~ \ref{alst}, $M'$ is an amalgamation of a latin square $L$. Assume that the cells of $L$ are partitioned by $F$. Then because $M'$ is the $(S,T,U)$-amalgamation of $L$, each region of $L$ contains all of the symbols $1, \dots, n$, and so $L$ is a realization of $F$. \end{proof}

\begin{figure}
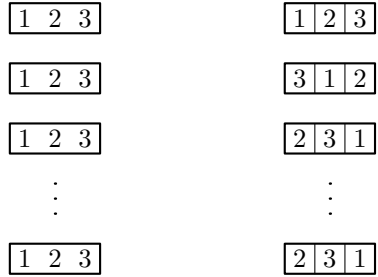

\centering \includegraphics{ger22.mps} \hspace{1cm} \hspace{1cm} \includegraphics{ger24.mps}
\caption{Permuting the slices of $R$ and vertically-aligned regions.} \label{slices2}
\end{figure}

\section{Gerechte Designs with rectangular regions: two~other~cases.}

In this section we shall consider two further cases. First, we give a general result about gerechte frameworks. An $n \times n$ array is said to be a \emph{row-latin square}, of order $n$, if each symbol appears exactly once in each row; and it is said to be a \emph{column-latin square} if each symbol appears exactly once in each column. (So a latin square is both a row-latin square and a column-latin square.)

Given a gerechte framework $F$ of order $n$, a \emph{row-realization} of $F$ is a row-latin square $W$, such that when the cells of $W$ are partitioned by $F$, each symbol appears exactly once in each region.

\begin{figure}
\centering \includegraphics{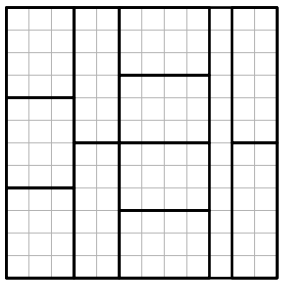}
\caption{A gerechte framework whose regions are arranged in columns.} \label{columns-fig}
\end{figure}

\begin{lemma} Every gerechte framework has a row-realization. \label{row-realization} \end{lemma}

\begin{proof} Suppose $F$ is a gerechte framework of order $n$. Let $G$ be the bipartite multigraph with vertex sets $\{\rho_1, \dots, \rho_n\}$ and $\{r_1, \dots, r_n\}$, representing respectively the rows and regions of $F$, where there are $l$ edges between $\rho_j$ and $r_k$ if $l$ cells of row $j$ belong to region $k$. Each vertex of $G$ has degree $n$, so by Theorem~\ref{konig}, we can give $G$ a proper edge-colouring with $n$ colours $c_1, \dots, c_n$. We can use this edge-colouring to create a row-latin square $W$ as follows: for each edge of colour $i$ between $\rho_j$ and $r_k$ we place symbol $i$ in one of the cells of row $j$ that belongs to region $k$ (which one does not matter). Since each vertex of $G$ is incident with exactly one edge of each colour, $W$ is a row-realization of $F$. \end{proof}

Suppose we have a gerechte framework $F$ with rectangular regions. If each region of $F$ is vertically aligned with all the regions above and below it, we say that its regions are \emph{arranged in columns}. For example, Figure~\ref{columns-fig} shows a gerechte framework of order 12 whose regions are arranged in columns.

\begin{theorem} Let $F$ be a gerechte framework whose regions are arranged in columns. Then $F$ is realizable. \label{aligned-in-columns} \end{theorem}

\begin{proof} By Lemma~\ref{row-realization}, $F$ has a row-realization $W$. Suppose that, going along a row of $F$ from left to right, the regions have widths $w_1, \dots, w_t$. We can form an outline latin square by amalgamating columns whose cells belong to the same set of regions, to form an $(S,T,U)$-outline latin square $O$, where $S = U = (1, \dots, 1)$ and $T = (w_1, \dots, w_t)$. In $O$, column $j$ will contain each symbol $w_j$ times. (So if $F$ is the example in Figure~\ref{columns-fig}, $O$ is an outline latin square with 5 columns and 12 rows.)
By Theorem~\ref{alst}, $O$ is an amalgamation of a latin square $L$. By construction, each slice of $L$ contains the same symbols as the corresponding slice in $W$, possibly rearranged, and since $W$ is a row-realization of $F$, each region of $L$ contains each symbol once. Hence $L$ is a realization of $F$. \end{proof}	

\begin{figure}
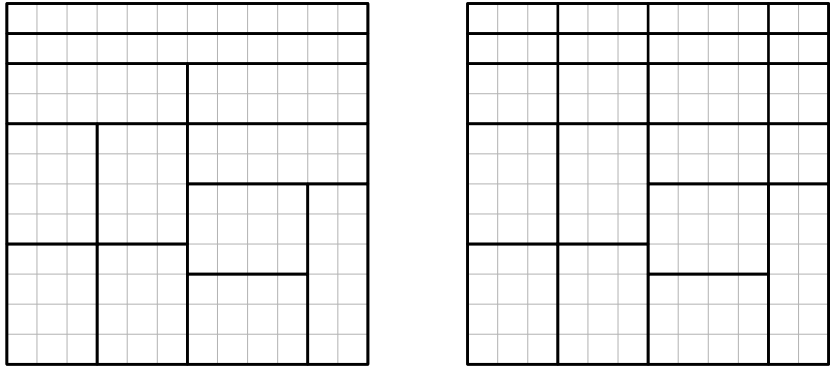

\centering \includegraphics{ger46.mps} \hspace{1cm} \includegraphics{ger46n.mps}
\caption{A gerechte framework $F$ whose regions are arranged in a tree structure (left) and the refined framework $F'$ (right).} \label{tree-fig}
\end{figure}

Suppose we have a gerechte framework $F$ with rectangular regions, and that $R$ is one of the regions. We define the \emph{top set} of $R$ to be the set of cells that are either in $R$, or are above $R$; and we define the \emph{bottom set} of $R$ to be the set of cells that are below $R$. (Note that the cells of $R$ are included in the top set, but not the bottom set.)

We say that a gerechte framework is \emph{arranged in a tree structure} if the bottom set of every region $R$ contains only complete regions. For example, Figure \ref{tree-fig} (left) shows a gerechte framework of order 12 whose regions are arranged in a tree structure.

If we have a gerechte framework $F$ whose regions are arranged in a tree structure, we can form what is called the \emph{refined framework} $F'$, by extending the vertical lines of $F$ upwards to the top of the square. Figure \ref{tree-fig} gives an example of a framework $F$ and its refined framework $F'$. Note that the refined framework is not in general a gerechte framework, as some regions may contain fewer than $n$ cells.

We can now give our final result. The proof is somewhat similar to that of Theorem \ref{aligned-in-columns}, in that we start with a row-realization and then construct an outline latin square, from which we can deduce that a realization exists.

\begin{theorem} Let $F$ be a gerechte framework whose regions are arranged in a tree structure. Then $F$ is realizable. \label{tree-structure} \end{theorem}

\begin{proof} We begin as in the proof of Theorem \ref{aligned-in-columns}, by choosing a row-realization $W$ of $F$. We shall describe the construction of another square $W'$, which will have the property that when it is partitioned by the refined framework $F'$, each set of vertically-aligned regions will contain equal numbers of each symbol. $W'$ will be constructed from $W$ by rearranging the symbols within its slices. As the rearrangement of symbols will occur within slices, $W'$ will also be a row-realization of $F$.

\begin{figure}
\centering \includegraphics{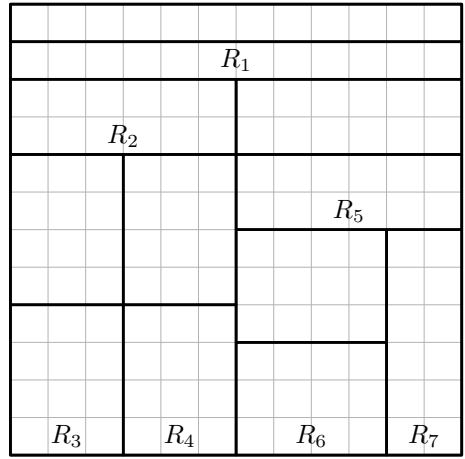}
\caption{The regions $R_1, \dots, R_7$ of the gerechte framework from Figure \ref{tree-fig}.} \label{regions1-7}
\end{figure}

We begin by partitioning the regions of $F$ into equivalence classes according to vertical alignment. Suppose there are $p$ such classes. We choose a set of representatives, one region from each class, by choosing the region from each class that is closest to the bottom of the square. (See Figure \ref{regions1-7}, where we have done this with the gerechte framework from Figure \ref{tree-fig}.) We can then order the representatives $R_1, \dots, R_p$ so that whenever a region $R_i$ is above a region $R_j$, $i < j$. (A simple way to do this is just to order the regions according to the row they begin in, starting with the regions that begin in row 1.)

Suppose the top sets of the regions $R_1, \dots, R_p$ are $T_1, \dots, T_p$.
We shall describe the construction of a sequence of row-latin squares $W_1, \dots, W_p$ that will each be row-realizations of $F$. The square $W_i$ will have the property that the top sets of regions directly below $T_i$ contain the same number of copies of each symbol. The final square, $W_p$, will be the desired $W'$.

Suppose the row-latin squares $W_1, \dots, W_{i-1}$ have been constructed already, and that we wish to construct $W_i$. We may assume that $T_i$ contains the same number of copies of each symbol. The idea is to take $W_{i-1}$, and rearrange the symbols within the slices of $T_i$, to create $W_i$.

We partition $T_i$ into parts, called \emph{chunks}, according to the region boundaries directly below $T_i$; the rule is that we extend the vertical lines of $F$ that divide the regions \emph{directly below} $T_i$ upwards through $T_i$ to the top of the square. (See Figure~\ref{chunks}.) (Note that it is important that we only do this for the boundaries of $F$ that are directly below $T_i$.)

Since the regions of $F$ are arranged in a tree structure, the set of cells below each chunk consists of complete regions of $F$, and so contains a number of cells that is a multiple of $n$. Since there are $n$ cells in each column, it follows that the number of cells in each chunk is itself a multiple of $n$.

We claim that we can rearrange the slices of $T_i$ so that each chunk contains the same number of copies of each symbol. Suppose that $T_i$ is of size $s \times t$, and that there are $m$ chunks, of widths $t_1, \dots, t_m$, and let
\[d = \gcd \{t_1, \dots, t_m\}.\]
Since $n \mid st_i$ for each $i \in \{1, \dots, m\}$, we have $n \mid sd$. It will suffice to show that we can rearrange the symbols within each row, so that if $T_i$ is partitioned into $q=t/d$ chunks of width $d$, each symbol appears the same number of times in each chunk. (Since the same will then also be true of the original chunks.)

\begin{figure}
\centering
\begin{tikzpicture}
\fill[gray!10] (0,0) rectangle (6,1.5);
\draw[gray, dashed] (0,3.0)--(0,1.5) (6,3.0)--(6,1.5)
(2,0)--(2,3.0) (4,0)--(4,3.0);
\draw[black, thick] (0,0)--(6,0) (0,0.5)--(6,0.5) (0,1.0)--(6,1.0)
(-0.2,1.5)--(6.2,1.5) (-0.2,3.0)--(6.2,3.0)
(0,-0.2)--(0,1.5) (6,-0.2)--(6,1.5)
(2,0)--(2,-0.2) (4,0)--(4,-0.2);
\draw[black] (3,0) node[above] {$R_i$};
\end{tikzpicture}
\caption{$T_i$ consists of $R_i$ and the regions vertically-aligned with it, and the cells above. The dashed lines indicate the partition into chunks.}
\label{chunks}
\end{figure}
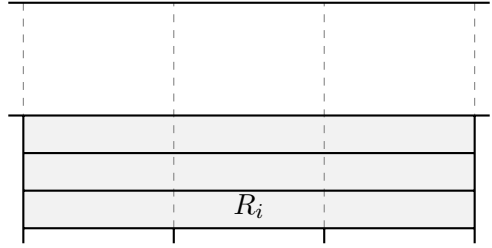

Let $G$ be the bipartite graph with vertex sets
\[\{\rho_1, \dots, \rho_s\} \text{\ and \ } \{\sigma_1, \dots, \sigma_n\},\]
representing respectively the rows of $T_i$ and the symbols $1, \dots, n$, where there is an edge between $\rho_j$ and $\sigma_k$ if row $j$ of $T_i$ contains the symbol $k$. Then each of the vertices $\rho_1, \dots, \rho_s$ has degree $t$, and each of the vertices $\{\sigma_1, \dots, \sigma_n\}$ has degree $r = st/n$.

By definition, $q \mid t$, and we have $q \mid r$, because $r = st/n = qds/n$, and $ds/n$ is an integer. So by Theorem~\ref{vkonig}, we can give $G$ an equitable edge-colouring with $q$ colours $c_1, \dots, c_q$. We can then rearrange the symbols in row $j$ of $T_i$ by placing the symbol $k$ in a cell in chunk $h$ if there is an edge of colour $c_h$ between $\sigma_k$ and $\rho_j$.

Once this procedure is complete, the top sets of regions directly below $T_i$ will contain the same number of copies of each symbol, and we will have constructed $W_i$. In this way we can construct the row-latin squares $W_1, \dots, W_p = W'$.

We can then partition $W'$ by the refined framework $F'$, and form an outline latin square by amalgamating columns whose cells belong to the same set of regions of $F'$, to form an $(S,T,U)$-outline latin square $O$, where $S = U = (1, \dots, 1)$ and $T = (w_1, \dots, w_u)$, where $w_1, \dots, w_u$ are the widths of the regions of $F'$. In $O$, column $j$ will contain each symbol $w_j$ times.

By Theorem \ref{alst}, $O$ is an amalgamation of a latin square $L$. By construction, each slice of $L$ contains the same symbols as the corresponding slice in $W$, possibly rearranged, and since $O$ is a row-realization of $F$, each region of $L$ contains each symbol once. Hence $L$ is a realization of $F$. \end{proof}

\section{Concluding remarks}

In the above, we considered gerechte frameworks whose regions are rectangles. It would be interesting to see if Theorem~\ref{main} is still true if we allow the rectangles to ``wrap around'' the edges. In other words, if we imagine the cells to be drawn on a torus, rather than a square.

However, the central question seems to be the following.

\begin{question} Let $F$ be a gerechte framework for which each region is a rectangle. If $F$ necessarily realizable? \end{question}

\section*{Acknowledgements}

The authors would like to thank Peter J. Cameron and A. J. W. Hilton for some helpful discussions.


\begin{thebibliography}{9}

\bibitem{bcc} R. A. Bailey, P. J. Cameron and R. Connelly. Sudoku, gerechte designs, resolutions, affine space, spreads, reguli, and Hamming codes, \textit{Amer. Math. Monthly} \textbf{115} (2008), 383--404.

\bibitem{bm} J. A. Bondy and U. S. R. Murty, \textit{Graph Theory}, Springer, 2008.

\bibitem{behrens} W. U. Behrens. Feldversuchsanordnungen mit verbessertem Ausgleich der Bodenunterschiede, \textit{Zeitschrift f\"ur Landwirtschaftliches Versuchsund Untersuchungswesen} \textbf{2} (1956), 176--193.

\bibitem{cgcs2007} P. J. Cameron. Problems from CGCS Luminy, May 2007, \textit{European J. Combin.} \textbf{31} (2010), 644--648.

\bibitem{hilton} A. J. W Hilton. \textit{Outlines of Latin Squares}, \emph{Ann. Discrete Math.} \textbf{34} (1987), 225--242.


\bibitem{vaughan} E. R. Vaughan. The complexity of constructing gerechte designs, \textit{Electron. J. Combin.}, \textbf{16} (2009), R15.

\bibitem{dewerra} D. de Werra. Equitable colorations of graphs, \textit{Rev. Fran\c{c}aise Informat. Recherche Op\'{e}rationnelle} \textbf{5} (1971), Ser. R-3, 3--8.

\end{thebibliography}
\end{document}